\documentclass[a4paper,10pt]{amsart}
\usepackage[utf8]{inputenc}
\usepackage{amsxtra}
\usepackage{amsopn}
\usepackage{amsmath,amsthm,amssymb}
\usepackage{amscd}
\usepackage{amsfonts}
\usepackage{latexsym}
\usepackage{verbatim}
\usepackage{MnSymbol}
\newcommand{\lr}{\longrightarrow}

\newcommand{\la}{\llangle}
\newcommand{\ra}{\rrangle}
\theoremstyle{plain}
\newtheorem{theorem}{Theorem}[section]
\newtheorem*{theorem*}{Theorem}

\newtheorem{lemma}[theorem]{Lemma}

\newtheorem{rem}[theorem]{Remark}

\newtheorem*{mt*}{Main Theorem}
\newcommand\C{{\mathbb C}}

\newcommand\N{{\mathbb N}}
\newcommand\R{{\mathbb R}}

\newcommand{\de}[2]{\frac{\partial #1}{\partial #2}}

\newcommand{\del}{\partial}
\newcommand{\delbar}{\overline{\del}}

\newcommand{\Cinf}{$\mathcal{C}^\infty$}
\DeclareMathOperator{\supp}{supp}

\let\sup\undefined
\DeclareMathOperator*{\sup}{sup\vphantom{p}}
\DeclareMathOperator{\vol}{Vol_g}
\DeclareMathOperator{\sgn}{sgn}
\DeclareMathOperator{\real}{Re}

\title{Bott-Chern Harmonic Forms on Stein Manifolds}
\author{Riccardo Piovani and Adriano Tomassini}
\address{Dipartimento di Scienze Matematiche, Fisiche ed Informatiche\\
Unit\`{a} di Matematica e Informatica\\
Universit\`{a} degli Studi di Parma\\
Parco Area delle Scienze 53/A, 43124\\
Parma, Italy}
\email{riccardo.piovani@studenti.unipr.it, adriano.tomassini@unipr.it}

\thanks{This work was partially supported by the Project PRIN ``Varietà reali e complesse: geometria, topologia e analisi armonica'' 
and by GNSAGA of INdAM}

\begin{document}
\maketitle

\begin{abstract}
Let $M$ be an $n$-dimensional $d$-{\em bounded Stein manifold $M$}, i.e., a complex 
$n$-dimensional manifold $M$ admitting a smooth strictly plurisubharmonic exhaustion $\rho$ and endowed with the K\"ahler metric whose fundamental 
form is $\omega=i\del\delbar\rho$, such that $i\delbar\rho$ has bounded $L^\infty$ norm. We prove a vanishing result for $W^{1,2}$ harmonic forms  
with respect to the Bott-Chern Laplacian on $M$.
\end{abstract}

\section{Introduction}
Let $M$ be an $n$-dimensional complex manifold endowed with a Hermitian metric $g$. Then, on the space of $(p,q)$-forms on $M$ there are defined 
several self-adjoint elliptic differential operators of order two, as the {\em Dolbeault Laplacian} $\Delta^g_{\delbar}$ and of order four as 
{\em Bott-Chern} and 
{\em Aeppli Laplacians}, respectively denoted by $\tilde{\Delta}^g_{BC}$ and $\tilde{\Delta}^g_{A}$, involving $\del,\delbar,\del^*,\delbar^*$ and their 
suitable combinations. If $M$ is compact, then, in particular, according to Schweitzer \cite{S}, it  turns out that a 
$(p,q)$-form $\varphi$ on $M$ satisfies $\tilde{\Delta}^g_{BC}\varphi=0$ 
if and only if 
$$\del\varphi=0,\quad\delbar\varphi=0,\quad \delbar^*\del^*\varphi=0.
$$
Furthermore, if the Hermitian metric $g$ is also K\"ahler on the compact complex manifold $M$, then according to the ellipticity, the kernels 
of such differential 
operators are finite dimensional and, as a consequence of K\"ahler identities, they coincide. 

For a complex non compact manifold $M$, one can consider smooth $L^2$ forms and study the space of $L^2$ harmonic $(p,q)$-forms for the above Laplacians. 
In \cite{DF}, the space of $L^2$ harmonic forms with respect to the Dolbeault Laplacian $\Delta^g_{\delbar}$ on a bounded strictly 
pseudoconvex domain $\Omega$ in $\C^n$ with 
smooth boundary, endowed with the Bergman metric, is studied. Precisely, denoting by $\mathcal{H}^{p,q}_{\delbar,2}$ the space 
of $L^2$ harmonic forms with respect to $\Delta^g_{\delbar}$, Donnelly and Fefferman proved that under the assumptions as above, it 
holds that
$$\left\{
\begin{array}{lll}
\dim\mathcal{H}^{p,q}_{\delbar,2}=0 &\hbox{if}&\,\,\, p+q\neq n\\[10pt]
\dim\mathcal{H}^{p,q}_{\delbar,2}=\infty&\hbox{if}&\,\,\, p+q= n.
\end{array}
\right.
$$
In \cite{O} Ohsawa proved that the dimension of the middle $L^2$ $\delbar$-cohomology of a domain in a complex manifold, admitting 
a non-degenerate regular boundary point and whose defining function satisfies some suitable assumptions, is infinite dimensional.
Later, Gromov in \cite{G} introduced the notion of K\"ahler hyperbolicity showing that, if $X$ is a K\"ahler 
complete simply-connected manifold whose K\"ahler 
form $\omega$ is $d$-bounded, admitting a uniform discrete subgroup of isometries, then 
$$\left\{
\begin{array}{lll}
\mathcal{H}^{p,q}_{\delbar,2}=\{0\} &\hbox{if}&\,\,\, p+q\neq n\\[10pt]
\mathcal{H}^{p,q}_{\delbar,2}\neq\{0\}&\hbox{if}&\,\,\, p+q= n.
\end{array}
\right.
$$

In the present paper we are interested in studying the space of $L^2$ $(p,q)$-harmonic forms with respect to the 
Bott-Chern Laplacian on Hermitian complete manifolds. More precisely, we consider a $d$-{\em bounded Stein manifold $M$}, i.e., a complex 
$n$-dimensional manifold $M$ admitting a smooth strictly plurisubharmonic exhaustion $\rho$ and endowed with the K\"ahler metric whose fundamental 
form is $\omega=i\del\delbar\rho$, such that $i\delbar\rho$ has bounded $L^\infty$ norm; examples of such manifolds 
are bounded strictly pseudoconvex domains in $\C^n$ with 
smooth boundary, endowed with the Bergman metric (see \cite{DO}). Denoting by $\mathcal{H}^{p,q}_{BC,2}$ the space of $W^{1,2}$ 
Bott-Chern harmonic $(p,q)$-forms, where $W^{1,2}$ denotes the Sobolev Space, we prove the following vanishing result\vskip.2truecm\noindent
{\bf Theorem} (see Theorem \ref{GromovBC}) {\em 
Let $M$ be a $d$-bounded Stein manifold of complex dimension $n$. Then 
$$\mathcal{H}^{p,q}_{BC,2}=\{0\},\quad \hbox{for}\,\, p+q\neq n.
$$
}
\vskip.2truecm\noindent
The paper is organized as follows: in Section \ref{preliminaries} we fix some notation and recall some well known results in K\"ahler geometry. 
In Section \ref{cut-off}, adapting an argument by Demailly \cite{DE}, we construct cut-off functions on a $d$-bounded Stein manifold, proving estimates 
of second order derivatives. Section \ref{pezzi} is mainly devoted to prove that a smooth $W^{1,2}$ $(p,q)$-form $\varphi$ satisfies 
$\tilde{\Delta}^g_{BC}\varphi=0$ if and only if 
$\del\varphi=0,\quad\delbar\varphi=0,\quad \delbar^*\del^*\varphi=0$ (see Theorem \ref{teo:pezzi}). Basic tools in the proof of 
such a theorem are the estimates of second order 
derivatives of cut-off functions ensured by Stein $d$-bounded assumption (see Lemma \ref{estimates-derivatives}). 
As a corollary, we also derive that (see Theorem \ref{teo:harmonic}), 
\begin{equation*}
\mathcal{H}^{p,q}_{BC,2}\subset\mathcal{H}^{p,q}_{\delbar,2}=\mathcal{H}^{p,q}_{\del,2}=\mathcal{H}^{p,q}_{d,2},
\end{equation*}
where the last two sets are the spaces of $L^2$-harmonic forms with respect to the $\del$-Laplacian $\Delta^g_{\del}$ and 
to the Hodge-de Rham Laplacian $\Delta^g_{d}$. Then, combining Theorem \ref{teo:harmonic} with the results by Gromov, we obtain the 
proof of the vanishing Theorem \ref{GromovBC}.

In \cite{G} Gromov gave an $L^2$ Hodge decomposition Theorem for complete Riemannian manifolds, (see also \cite[Chap.VIII, Thm.3.2]{DE}). 
It seems to be harder to prove a link between $W^{1,2}$ 
Bott-Chern harmonic forms and $L^2$ reduced cohomology. For other results on $L^2 $ cohomological decomposition in the complete almost Hermitian setting 
see \cite{HT}. Thanks to elliptic regularity, we consider only $L^2$ forms that are also smooth. 
It is sufficient for our results.

\medskip

\noindent{\sl Acknowledgments.} We are grateful to Professor Boyong Chen for useful suggestions and remarks for a
better presentation of the results.
\section{Preliminaries}\label{preliminaries}
Let $M$ be an $n$-dimensional complex manifold. Denote by $\Omega^r(M;\C)$, respectively $\Omega^{p,q}(M)$ the space of smooth complex $r$-forms, 
respectively 
smooth $(p,q)$-forms on $M$. Let $g$ be a Hermitian metric on $M$; denote by $\omega$ the fundamental form of $g$ and by 
$\hbox{Vol}_g:=\frac{\omega^n}{n!}$ the 
standard volume form. Let $\langle,\rangle$ be the pointwise Hermitian inner product induced by $g$ on the space of $(p,q)$-forms. 
Given any $\varphi\in\Omega^r(M;\C)$, set 
$$\vert\varphi(x)\vert^2_g=\langle\varphi,\varphi\rangle(x)$$
and
\begin{gather*}
\lVert \varphi\rVert_{L^2}:=\int_M\vert\varphi\vert^2_g\hbox{Vol}_g,\\
\lVert \varphi\rVert_{W^{1,2}}:=\int_M(\vert\varphi\vert^2_g+\vert\nabla\varphi\vert^2_g)\hbox{Vol}_g,
\end{gather*}
where $\nabla$ is the Levi-Civita connection of $(M,g)$ and $\vert\nabla\varphi\vert_g$ is the pointwise Hermitian norm of the 
covariant derivative of the $(p,q)$-form $\varphi$, induced by $g$ on the space of complex covariant tensors on $M$.
Define
$$
L^2(M):=\left\{\varphi\in\Omega^r(M;\C)\,\,\,\vert\,\,\, 0\le r\le 2n,\ \lVert \varphi\rVert_{L^2}<\infty \right\}
$$
and
$$
W^{1,2}(M):=\left\{\varphi\in\Omega^r(M;\C)\,\,\,\vert\,\,\, 0\le r\le 2n,\ \lVert \varphi\rVert_{W^{1,2}}<\infty \right\}
$$
For any given $\varphi\in\Omega^r(M;\C)$, we also set 
$$
\lVert \varphi\rVert_{L^\infty}:=\sup_{x\in M}\vert\varphi(x)\vert_g
$$
and we call $\varphi$ {\em bounded} if $
\lVert \varphi\rVert_{L^\infty}<\infty$. Furthermore, if $\varphi=d\eta$, then $\varphi$ is said to be {\em d-bounded}, if $\eta$ is bounded. 
For any $\varphi,\psi\in\Omega_c^{p,q}(M)$ denote by $\la,\ra$ the $L^2$-Hermitian product defined as
$$
\la\varphi,\psi\ra =\int_M\langle\varphi,\psi\rangle \hbox{Vol}_g
$$\newline
Denoting by $*:\Omega^{p,q}(M)\to \Omega^{n-p,n-q}(M)$ the complex anti-linear Hodge operator 
associated with $g$, the {\em Bott-Chern Laplacian} and {\em Aeppli Laplacian} $\tilde\Delta_{BC}^g$ and $ \tilde\Delta_{A}^g$ 
are the $4$-th order elliptic 
self-adjoint differential operators defined respectively as (see \cite[p.8]{S})
$$ 
\tilde\Delta_{BC}^g \;:=\;
\del\delbar\delbar^*\del^*+
\delbar^*\del^*\del\delbar+\delbar^*\del\del^*\delbar+
\del^*\delbar\delbar^*\del+\delbar^*\delbar+\del^*\del $$
and
$$ \tilde\Delta_{A}^g \;:=\; \del\del^*+\delbar\delbar^*+
\delbar^*\del^*\del\delbar+\del\delbar\delbar^*\del^*+
\del\delbar^*\delbar\del^*+\delbar\del^*\del\delbar^*\,,
$$
where, as usual
$$
\del^*=-*\del\, *\,,\qquad \delbar^*=-*\delbar\, *.
$$
\begin{rem}\label{bc-kh-id}{If $g$ is a K\"ahler metric on $M$, then as a consequence of the K\"ahler identities, (see \cite[p.120]{H} 
and \cite[Prop.2.4]{S}), it is
$$
\tilde\Delta_{BC}^g=\Delta^g_{\bar\partial}\Delta^g_{\bar\partial}+\delbar^*\delbar+\del^*\del,\quad 
\tilde\Delta_{A}^g=\Delta^g_{\bar\partial}\Delta^g_{\bar\partial}+\del\del^*+\delbar\delbar^*,
$$
$\Delta^g_{\bar\partial}$ being the Dolbeault Laplacian on $(M,g)$. }
\end{rem}
Finally, we define the space of $W^{1,2}$ Bott-Chern harmonic forms by setting 
$$
\mathcal{H}^{p,q}_{BC,2}:=\left\{\varphi\in\Omega^{p,q}(M)\,\,\,\vert\,\,\,\tilde\Delta_{BC}^g\varphi=0,\ \varphi\in W^{1,2}(M)\right\}.
$$

In the sequel we will need a local expression for the operators $\delbar^*$. To this purpose, let $(z^1,\ldots, z^n)$ 
be local holomorphic coordinates on $M$,  
$$
g=\sum_{\alpha,\beta}g_{\alpha\bar\beta}dz^\alpha\otimes d\bar{z}^\beta
$$
be the local expression of the Hermitian metric $g$ and $(g^{\bar\alpha\beta})=(g_{\bar\alpha\beta})^{-1}$. 
Given $\psi\in\Omega^{p,q}(M)$, for 
$$
A_p=(\alpha_1,\ldots,\alpha_p),\qquad 
B_q=(\beta_1,\ldots,\beta_q)$$ 
multiindices of length $p$, $q$ respectively, with 
$\alpha_1<\cdots <\alpha_p$ and $\beta_1<\cdots <\beta_q$, denote by 
$$
\psi=\sum_{A_p, B_q}\psi_{A_p\bar{B}_q}dz^{A_p}\wedge d\bar{z}^{B_q}
$$
the local expression of $\psi$ and by
$$
\psi^{\bar{A}_pB_q}=\sum_{\Gamma_p,\Lambda_q}g^{\bar{\alpha}_1\gamma_1}\cdots g^{\bar{\alpha}_p\gamma_p}
g^{\bar{\lambda}_1\beta_1}\cdots g^{\bar{\lambda}_q\beta_q}\psi_{\gamma_1\ldots \gamma_p\bar{\lambda}_1\ldots \bar{\lambda}_q}.
$$
Then, if $\varphi\in\Omega^{p,q}(M)$, locally the pointwise Hermitian inner product $\langle,\rangle$ on $\Omega^{p,q}(M)$ is given by
$$
\langle \varphi,\psi\rangle=\sum_{A_p, B_q}\varphi_{A_p\bar{B}_q}\overline{\psi^{\bar{A}_pB_q}}.
$$
The local expression of the pointwise Hermitian inner product induced by $g$ on the space of complex covariant tensors on $M$ is similar.

According to \cite[Prop.2.3]{MK}, we recall the local formula for $\delbar^*$, that is for any given $\psi\in\Omega^{p,q+1}(M)$, it is
$$
(\delbar^*\psi)^{\bar{A}_pB_q}=-\sum_{\gamma=1}^n
\left(\frac{\partial}{\partial z^\gamma}+\frac{\partial\log \det(g_{\alpha\bar\beta})}{\partial z^\gamma}\right)\psi^{\gamma\bar{A}_pB_q}.
$$
In the special case that $g$ is a K\"ahler metric on $M$, then as a consequence of the last formula, for every fixed $x_0\in M$, denoting by 
$(z^1,\ldots, z^n)$ local normal holomorphic coordinates at $x_0$, we obtain
\begin{equation}\label{delbar-normal}
(\delbar^*\psi)_{\alpha_1\ldots\alpha_p\bar{\beta}_1\ldots\bar{\beta}_q}(x_0)=-\sum_{\gamma=1}^n
\frac{\partial \psi_{\bar{\gamma}\alpha_1\ldots\alpha_p\bar{\beta}_1\ldots\bar{\beta}_q}}{\partial z^\gamma}(x_0).
\end{equation}
Finally, for any given $\varphi\in\Omega^{p,q}(M)$, still using local normal holomorphic coordinates at $x_0$, we have
\begin{equation}\label{nabla-normal}
|\nabla\varphi|^2_g(x_0)=2\sum_{A_p, B_q}\sum_{\gamma=1}^n
\biggl(\biggl|\de{\varphi_{A_p\bar{B}_q}}{z^\gamma}\biggr|^2+\biggl|\de{\varphi_{A_p\bar{B}_q}}{\bar{z}^\gamma}\biggr|^2\biggl)(x_0).
\end{equation}

\section{Construction of cut-off functions in $d$-bounded Stein manifolds}
\label{cut-off}
Let $M$ a be Stein manifold and let $\rho$ be a strictly plurisubharmonic exhausting smooth function. Denote $\omega=i\del\delbar\rho$ the 
fundamental form with the K\"ahler metric $g$ associated. We say that $M$ is {\em d-bounded} if $\omega=d\eta$ and $\eta=i\delbar\rho$ is 
bounded. In particular, $\omega$ is $d$-bounded.  In the following, $\rho$, $\omega$, $g$, $\eta$ are considered fixed.

We remark that any $d$-bounded Stein manifold is complete, see \cite[Chap.VIII, Lemma 2.4]{DE}.

Examples of $d$-bounded Stein manifolds are bounded strictly pseudoconvex domains in $\C^n$ with smooth boundary, endowed with the Bergman metric, 
see \cite[Prop.3.4]{DO}.

Now we prove the existence of cut-off functions with specific bounds on the second order derivatives on a d-bounded Stein manifold. 
We need the following known lemma.

\begin{lemma}
Let $a,b\in\R$, $a<b$. Then there exists a \Cinf function $\psi:\R\rightarrow [0,1]\subset\R$ such that the following properties hold:
\begin{itemize}
\item $\psi(t)=1 \ \iff\  t\le a$;
\item $\psi(t)=0 \ \iff\  t\ge b$;
\item $\exists C\in\R$ such that $|\psi'(t)|,|\psi''(t)|\le C\psi(t)^{\frac{1}{2}}\ \forall t\in\R$.
\end{itemize}
\end{lemma}
\begin{proof}
Let us define $\phi:\R\rightarrow\R$, a \Cinf function such that
\begin{equation*}
\phi(t)=
\begin{cases}
\exp(-\frac{1}{t^2}) & \text{ if } t>0\\
0 & \text{ if } t\le 0.
\end{cases}
\end{equation*}
Then we define
\begin{equation*}
\psi(t)=\frac{\phi(b-t)}{\phi(b-t)+\phi(t-a)}.
\end{equation*}
Note that $\psi(t)=1$ iff $t\le a$ and $\psi(t)=0$ iff $t\ge b$. After some calculations we obtain
\begin{equation*}
\begin{split}
\psi'(t) & = \frac{-2\frac{\phi(b-t)}{(b-t)^3}(\phi(b-t)+\phi(t-a))-
\phi(b-t)\biggl(-2\frac{\phi(b-t)}{(b-t)^3}+2\frac{\phi(t-a)}{(t-a)^3}\biggr)}{(\phi(b-t)+\phi(t-a))^2}\\
& = -2\frac{\phi(b-t)\phi(t-a)}{(\phi(b-t)+\phi(t-a))^2}\biggl(\frac{1}{(b-t)^3}+\frac{1}{(t-a)^3}\biggr)\\
& = -2\psi(t)^\frac{1}{2}\frac{\phi(b-t)^\frac{1}{2}\phi(t-a)}{(\phi(b-t)+\phi(t-a))^\frac{3}{2}}
\biggl(\frac{1}{(b-t)^3}+\frac{1}{(t-a)^3}\biggr).
\end{split}
\end{equation*}
This implies that $\exists C\in\R$ such that $|\psi'(t)|\le C\psi(t)^{\frac{1}{2}}\ \forall t\in\R$. The 
calculations of the estimate on $\psi''$ are analogous.
\end{proof}

The following lemma is inspired by \cite[Chap.VIII, Lemma 2.4]{DE}.

\begin{lemma}\label{estimates-derivatives}
Let $M$ be a d-bounded Stein manifold of complex dimension $n$. Then there exists a sequence $\{K_\nu\}_{\nu\in\N}$ of 
compact subsets of $M$ and a sequence $a_\nu : M \rightarrow [0,1]\subset\R$, $\nu \in \N$, of \Cinf functions with compact support, 
called cut-off functions, such that the following properties hold:
\begin{itemize}
\item $\bigcup_{\nu\in\N} K_\nu=M$ and $K_\nu\subset \mathop K\limits^ \circ$$_{\nu+1}$;
\item $\forall\nu \in \N$ $a_\nu=1$ in a neighbourhood of $K_\nu$ 
and $\supp{a_\nu}\subset \mathop K\limits^ \circ$$_{\nu+1}$;
\item $\exists C\in\R$ such that $|\del a_\nu(x)|_g,|\delbar a_\nu(x)|_g,|\del\delbar a_\nu(x)|_g\le 2^{-\nu}C a_\nu(x)^\frac{1}{2}$ $\forall x\in M$.
\end{itemize}
\end{lemma}
\begin{proof}
We define
\begin{equation*}
a_\nu(x)=\psi(2^{-\nu}\rho(x)) \ \forall x\in M \text{ and } K_\nu=\overline{\{x\in M\ |\ \rho(x)< 2^{\nu}\}},
\end{equation*}
where $\psi$ is the function of the previous lemma, with $a=1.1$ and $b=1.9$. 

Let us check that the claimed properties hold.
The subsets $K_\nu$ are compact because of the definition of exhausting function. If $x\in M$, 
then $\exists\nu\in\N$ such that $\rho(x)<2^\nu$, so $x\in K_\nu$, thus $\bigcup_{\nu\in\N} K_\nu=M$. 
The inclusions $K_\nu\subset \mathop K\limits^ \circ$$_{\nu+1}$ hold by the construction of $K_\nu$, in fact if $x\in K_\nu$, then $\rho(x)\le 2^\nu$ 
by continuity and $x\in\{y\in M\ |\ \rho(y)<2^\nu\cdot 1.5\}\subset\mathop K\limits^ \circ$$_{\nu+1}$.

The functions $a_\nu$ are \Cinf because $\psi$ and $\rho$ are \Cinf. The function $\psi$ takes values in 
the interval $[0,1]$, and so is true for $a_\nu$. Choosing a sufficiently small neighbourhood of $K_\nu$, we can assume that $2^{-\nu}\rho<1.1$, 
so $a_\nu=1$ in that neighbourhood. 

In order to prove $\supp{a_\nu}\subset \mathop K\limits^ \circ$$_{\nu+1}$, let us take $\tilde{x}\in\supp{a_\nu}$ 
and a sequence $\{x_k\}_{k\in\N}$ of points in $M$ such that $a_\nu(x_k)>0\ \forall k\in \N$ 
and $x_k\rightarrow\tilde{x}$ as $k\rightarrow\infty$. By the construction of $\psi$, 
we have that, $\forall x\in M$, $a_\nu(x)>0$ if and only if $2^{-\nu}\rho(x)< 1.9$. 
Therefore, by the continuity of $\rho$, $2^{-\nu}\rho(\tilde{x})\le 1.9$, so that $\rho(\tilde{x})< 2^{\nu}\cdot 1.95$ 
and $\tilde{x}\in \mathop K\limits^ \circ$$_{\nu+1}$. Because $\supp{a_\nu}$ is a close set contained in a compact set, then it is compact.

Finally we have to prove the estimates on the differentials of $a_\nu$. Let $x\in M$, then
\begin{equation*}
\del a_\nu(x)= 2^{-\nu}\psi'(2^{-\nu}\rho(x)) \sum\limits_{i=1}^n \de{\rho}{z^i}(x)dz^i,
\end{equation*}
and
\begin{equation*}
\begin{split}
|\del a_\nu(x)|_g & = 2^{-\nu}|\psi'(2^{-\nu}\rho(x))| |\del\rho(x)|_g \\
& \le 2^{-\nu}C(\psi(2^{-\nu}\rho(x)))^{\frac{1}{2}} ||\del\rho||_{L^\infty} \\
& \le 2^{-\nu}Ca_\nu(x)^{\frac{1}{2}},
\end{split}
\end{equation*}
where the constant $C$ is taken big enough 
and may not be the same at every passage of the calculations. In the last passage we used the 
hypothesis that $\omega$ is $d$-bounded and the fact that $\rho$ is real, 
so $\de{\rho}{\bar{z}^i}(x)=\overline{\de{\rho}{z^i}}(x)$ and $|\del\rho(x)|_g=|\delbar\rho(x)|_g$. 
By the same calculations, we also obtain the estimate of $|\delbar a_\nu(x)|_g$. Moreover,
\begin{equation*}
\begin{split}
\delbar\del a_\nu(x) & = 2^{-\nu} \sum\limits_{i,j=1}^n \biggl(2^{-\nu}\psi''(2^{-\nu}\rho(x)) \de{\rho}{\bar{z}^j}(x) \de{\rho}{z^i}(x)+\\  
& + \psi'(2^{-\nu}\rho(x)) \de{^2\rho}{\bar{z}^j \del z^i}(x)\biggr) d\bar{z}^j\land dz^i,
\end{split}
\end{equation*}
and
\begin{equation*}
\begin{split}
|\delbar\del a_\nu(x)|_g & \le 2^{-\nu} \biggl(2^{-\nu}|\psi''(2^{-\nu}\rho(x))||\delbar\rho(x)|_g|\del\rho(x)|_g+\\
& + |\psi'(2^{-\nu}\rho(x))| |\sum\limits_{i,j=1}^n \de{^2\rho}{\bar{z}^j \del z^i}(x) d\bar{z}^j\land dz^i|_g \biggr)\\
& \le 2^{-\nu}C(\psi(2^{-\nu}\rho(x)))^{\frac{1}{2}} \biggl(2^{-\nu}||\delbar\rho||_{L_\infty}||\del\rho||_{L_\infty}+|\omega(x)|_g\biggr)\\
& \le 2^{-\nu}Ca_\nu(x)^{\frac{1}{2}}.
\end{split}
\end{equation*}
In fact $\omega$ is $d$-bounded as before and $|\omega(x)|_g$ is constant, due to the definition of the metric. The proof is complete.
\end{proof}

\section{Vanishing of $L^2$ Bott-Chern harmonic forms}\label{pezzi}
\label{vanishing}
Our main theorem states that the following property, true in the compact case, also holds in our non-compact case.
\begin{theorem}
\label{teo:pezzi}
Let $M$ be a d-bounded Stein manifold of complex dimension $n$. Let $\varphi\in \Omega^{p,q}(M)\cap W^{1,2}(M)$.
If $\tilde{\Delta}_{BC}^g\varphi=0$, then 
\begin{equation*}
\del\varphi=0,\quad \delbar\varphi=0,\quad \delbar^*\del^*\varphi=0.
\end{equation*}
\end{theorem}
The following Lemma will be useful for the proof of Theorem \ref{teo:pezzi}.
\begin{lemma}\label{estimates}
Let $M$ be a K\"ahler manifold of complex dimension $n$ and denote its metric with $g$. 
If $\varphi\in\Omega^{p,q}(M)$ and $\{a_\nu\}_\nu$ are \Cinf functions on $M$, then $\exists C>0$ such that 
$\forall\nu\in\N$ 
\begin{equation}
\begin{split}
\label{lemma:delbar}
&|\delbar^*(\delbar a_\nu\land*\varphi)|_g\le C\bigl(|\del\delbar a_\nu|_g|\varphi|_g+|\delbar a_\nu|_g|\nabla\varphi|_g\bigr),\\
&|\delbar^*(\delbar a_\nu\land\varphi)|_g \le C\bigl(|\del\delbar a_\nu|_g|\varphi|_g+|\delbar a_\nu|_g|\nabla\varphi|_g\bigr).
\end{split}
\end{equation}
\end{lemma}
\begin{proof}
The pointwise Hermitian norm on forms is invariant under change of coordinates, 
so we can prove the inequalities locally with a uniform constant $C$. For every fixed $x_0\in M$, denoting by $(z^1,\ldots, z^n)$ 
local normal holomorphic coordinates at $x_0$, we have
\begin{gather*}
\delbar a_\nu=\sum_\beta \de{a_\nu}{\bar{z}^\beta}d\bar{z}^\beta, \\
\varphi=\sum_{A_p,B_q} \varphi_{A_p\bar{B}_q} dz^{A_p}\land d\bar{z}^{B_q}, \\
\delbar a_\nu\land\varphi =\sum_{A_p,B_q,\beta} \de{a_\nu}{\bar{z}^\beta} \varphi_{A_p\bar{B}_q} 
d\bar{z}^\beta\land dz^{A_p}\land d\bar{z}^{B_q}
\end{gather*}
We can write
\begin{equation*}
\delbar a_\nu\land\varphi= \frac{1}{p!(q+1)!}\sum_{\substack{\alpha_1,\dots,\alpha_p,\\ \beta_0,\beta_1,\dots,\beta_q}} 
(\delbar a_\nu\land\varphi)_{\alpha_1\dots\alpha_p\bar{\beta}_0\bar{\beta}_1\dots\bar{\beta}_q}
dz^{\alpha_1\dots\alpha_p}\land d\bar{z}^{\beta_0\beta_1\dots\beta_q}
\end{equation*}
where the coefficients $(\delbar a_\nu\land\varphi)_{\alpha_1\dots\alpha_p\bar{\beta}_0\bar{\beta}_1\dots\bar{\beta}_q}$ are antisymmetric 
in the indices $\alpha_1,\dots,\alpha_p,\bar{\beta}_0,\bar{\beta}_1,\dots,\bar{\beta}_q$, so 
\begin{equation*}
\begin{split}
(\delbar a_\nu\land\varphi)_{\bar{\beta}_0\alpha_1\dots\alpha_p\bar{\beta}_1\dots\bar{\beta}_q} & =
(-1)^p (\delbar a_\nu\land\varphi)_{\alpha_1\dots\alpha_p\bar{\beta}_0\bar{\beta}_1\dots\bar{\beta}_q} \\
& = \sum_{j=0}^q (-1)^{qj} \de{a_\nu}{\bar{z}^{\beta_j}}  \varphi_{\alpha_1\dots\alpha_p\bar{\beta}_{j+1}\dots\bar{\beta}_{j+q}},
\end{split}
\end{equation*}
where we set $\beta_{j+q}:=\beta_{j-1}$ for $j=1,\dots,q$.
Now we apply (\ref{delbar-normal}) and obtain
\begin{equation*}
\begin{split}
(\delbar^*(\delbar a_\nu\land\varphi))&_{\alpha_1\ldots\alpha_p\bar{\beta}_1\ldots\bar{\beta}_q}(x_0) = \\
& = -\sum_{\beta_0=1}^n
\de{}{z^{\beta_0}} \bigl( (\delbar a_\nu\land\varphi)_{\bar{\beta}_0\alpha_1\ldots\alpha_p\bar{\beta}_1\ldots\bar{\beta}_q}\bigr)(x_0)\\
& = -\sum_{\beta_0=1}^n \sum_{j=0}^q \de{}{z^{\beta_0}} \bigl(
(-1)^{qj} \de{a_\nu}{\bar{z}^{\beta_j}}  \varphi_{\alpha_1\dots\alpha_p\bar{\beta}_{j+1}\dots\bar{\beta}_{j+q}}
\bigr) (x_0) \\
& = -\sum_{\beta_0=1}^n \sum_{j=0}^q (-1)^{qj} \bigl(
\de{^2a_\nu}{z^{\beta_0}\partial\bar{z}^{\beta_j}}\varphi_{\alpha_1\dots\alpha_p\bar{\beta}_{j+1}\dots\bar{\beta}_{j+q}}(x_0)+\\
& \quad\quad\quad\quad\quad\quad\quad\quad
+ \de{a_\nu}{\bar{z}^{\beta_j}}\de{\varphi_{\alpha_1\dots\alpha_p\bar{\beta}_{j+1}\dots\bar{\beta}_{j+q}}}{z^{\beta_0}}(x_0)
\bigr) .
\end{split}
\end{equation*}
This yields
\begin{equation*}
|\delbar^*(\delbar a_\nu\land\varphi)|_g(x_0) \le C\bigl(|\del\delbar a_\nu|_g(z_0)|\varphi|_g(x_0)+|\delbar a_\nu|_g(x_0)|\nabla\varphi|_g(x_0)\bigr),
\end{equation*}
where $C$ depends only on $n,p,q$. To prove it, we have to do some calculations. We set
\begin{equation*}
\gamma_{\beta_0j}:=\de{^2a_\nu}{z^{\beta_0}\partial\bar{z}^{\beta_j}}\varphi_{\alpha_1\dots\alpha_p\bar{\beta}_{j+1}\dots\bar{\beta}_{j+q}}
\text{ and }
\lambda_{\beta_0j}:=\de{a_\nu}{\bar{z}^{\beta_j}}\de{\varphi_{\alpha_1\dots\alpha_p\bar{\beta}_{j+1}\dots\bar{\beta}_{j+q}}}{z^{\beta_0}},
\end{equation*}
with $\alpha_1,\dots,\alpha_p,\beta_1,\dots,\beta_q=1,\dots,n$, $\beta_0=1,\dots,n$ and $j=0,\dots,q$.
So we have, using (\ref{nabla-normal}),

\begin{equation*}
\begin{split}
|\delbar^*(\delbar a_\nu\land\varphi)|_g^2(x_0) & =
\frac{1}{p!q!}\sum_{\substack{\alpha_1,\dots,\alpha_p,\\ \beta_1,\dots,\beta_q}} \sum_{\beta_0,\beta'_0=1}^n \sum_{j,j'=0}^q
(\gamma_{\beta_0j}+\lambda_{\beta_0j})(\bar{\gamma}_{\beta'_0j'}+\bar{\lambda}_{\beta'_0j'})(x_0) \\
& = \frac{1}{p!q!}\sum (\gamma_{\beta_0j}\bar{\gamma}_{\beta'_0j'}+\gamma_{\beta_0j}\bar{\lambda}_{\beta'_0j'}+\lambda_{\beta_0j}
\bar{\gamma}_{\beta'_0j'}+\lambda_{\beta_0j}\bar{\lambda}_{\beta'_0j'})(x_0) \\
& = \frac{1}{p!q!}\sum
(\real(\gamma_{\beta_0j}\bar{\gamma}_{\beta'_0j'})+2\real(\gamma_{\beta_0j}\bar{\lambda}_{\beta'_0j'})+ 
\real(\lambda_{\beta_0j}\bar{\lambda}_{\beta'_0j'}))(x_0) \\
& \le \frac{2}{p!q!}\sum (|\gamma_{\beta_0j}|^2+|\lambda_{\beta_0j}|^2)(x_0) \\
& \le C \bigl(|\del\delbar a_\nu|_g(x_0)|\varphi|_g(x_0)+|\delbar a_\nu|_g(x_0)|\nabla\varphi|_g(x_0)\bigr)^2.
\end{split}
\end{equation*}

To prove the other inequality in (\ref{lemma:delbar}), we want to estimate $|\delbar^*(\delbar a_\nu\land*\varphi)|_g(x_0)$, 
so the calculations are analogous except for the fact that we use the following formula (see \cite[p.94]{MK}) for the Hodge star operator:
\begin{equation*}
*\varphi=(i)^n(-1)^{\frac{1}{2}n(n-1)+qn}\sum_{A_p,B_q} \sgn A\sgn B \det(g_{h\bar{k}}) \bar{\varphi}^{\bar{B}_qA_p}
dz^{A_{n-p}}\land d\bar{z}^{B_{n-q}},
\end{equation*}
where $A$ and $B$ are respectively the permutations that send $(1,\dots,n)$ in $(A_p,A_{n-p})$ and in $(B_q,B_{n-q})$.
\end{proof}

\begin{proof}[Proof of Theorem \ref{teo:pezzi}]
First of all, by remark \ref{bc-kh-id}, we have 
$$\tilde\Delta_{BC}^g=\Delta^g_{\bar\partial}\Delta^g_{\bar\partial}+\delbar^*\delbar+\del^*\del.
$$
Thanks to the cut-off functions of the previous lemma, now we can integrate by part, using the Stokes Theorem.
\begin{equation*}
\begin{split}
0 & = \la\tilde{\Delta}^g_{BC}\varphi,a_\nu\varphi\ra \\
  & =\la\delbar\delbar^*\delbar\delbar^*\varphi,a_\nu\varphi\ra+\la\delbar^*\delbar\delbar^*\delbar\varphi,a_\nu\varphi\ra+
	\la\delbar^*\delbar\varphi,a_\nu\varphi\ra+\la\del^*\del\varphi,a_\nu\varphi\ra \\
  & =\la\delbar\delbar^*\varphi,\delbar\delbar^*(a_\nu\varphi)\ra+\la\delbar^*\delbar\varphi,\delbar^*\delbar(a_\nu\varphi)\ra+
	\la\delbar\varphi,\delbar(a_\nu\varphi)\ra+\la\del\varphi,\del(a_\nu\varphi)\ra
\end{split}
\end{equation*}
Now we calculate every differential in the right sides of the inner products. 
\begin{equation*}
\begin{split}
\delbar\delbar^*(a_\nu\varphi) & = -\delbar*(\delbar a_\nu\land*\varphi+a_\nu\delbar*\varphi) \\
  & =(-1)^{p+q}*\delbar^*(\delbar a_\nu\land*\varphi) +\delbar a_\nu\land \delbar^*\varphi+a_\nu\delbar\delbar^*\varphi \\
\delbar^*\delbar(a_\nu\varphi) & = \delbar^*(\delbar a_\nu\land\varphi+a_\nu\delbar\varphi) \\
  & =\delbar^*(\delbar a_\nu\land\varphi) -*(\delbar a_\nu\land *\delbar\varphi)+a_\nu\delbar^*\delbar\varphi \\
\delbar(a_\nu\varphi) & =\delbar a_\nu\land\varphi+a_\nu\delbar\varphi \\
\del(a_\nu\varphi) & =\del a_\nu\land\varphi+a_\nu\del\varphi
\end{split}
\end{equation*}
Therefore,
\begin{equation*}
0 = \la\tilde{\Delta}^g_{BC}\varphi,a_\nu\varphi\ra= I_1(\nu)+I_2(\nu),
\end{equation*}
where
\begin{equation*}
I_1(\nu)= \int_M a_\nu\bigl(|\delbar\delbar^*\varphi|_g^2+|\delbar^*\delbar\varphi|_g^2+
|\delbar\varphi|_g^2+|\del\varphi|_g^2\bigr)\,\vol,
\end{equation*}
and
\begin{equation*}
\begin{split}
I_2(\nu) & = \int_M \bigl(
\langle\delbar\delbar^*\varphi,(-1)^{p+q}*\delbar^*(\delbar a_\nu\land*\varphi) +\delbar a_\nu\land \delbar^*\varphi\rangle+\\
& + \langle\delbar^*\delbar\varphi,\delbar^*(\delbar a_\nu\land\varphi) -*(\delbar a_\nu\land *\delbar\varphi)\rangle+\\
& + \langle\delbar\varphi,\delbar a_\nu\land\varphi\rangle+\langle\del\varphi,\del a_\nu\land\varphi\rangle\bigr)\,\vol.
\end{split}
\end{equation*}
We have $I_1(\nu)=|I_2(\nu)|$ and, by the monotone convergence theorem, as $\nu\rightarrow\infty$,
\begin{equation*}
I_1(\nu)\rightarrow \int_M \bigl(|\delbar\delbar^*\varphi|_g^2+|\delbar^*\delbar\varphi|_g^2+
|\delbar\varphi|_g^2+|\del\varphi|_g^2\bigr)\,\vol.
\end{equation*}
Thus, if we show that $I_1(\nu)\rightarrow 0$ as $\nu\rightarrow\infty$, we have 
\begin{equation}\label{I1}
\del\varphi=0,\qquad \delbar\varphi=0,\qquad\delbar\delbar^*\varphi=0,\qquad\delbar^*\delbar\varphi=0.
\end{equation}
Estimating $|I_2(\nu)|$, we obtain
\begin{equation*}
\begin{split}
|I_2(\nu)| & \le \int_M \bigl(
   |\delbar\delbar^*\varphi|_g(|\delbar^*(\delbar a_\nu\land*\varphi)|_g+|\delbar a_\nu|_g|\delbar^*\varphi|_g)+\\
& +|\delbar^*\delbar\varphi|_g(|\delbar^*(\delbar a_\nu\land\varphi)|_g+ |\delbar a_\nu|_g|\delbar\varphi|_g)+\\
& +|\delbar\varphi|_g|\delbar a_\nu|_g|\varphi|_g+|\del\varphi|_g|\del a_\nu|_g|\varphi|_g\bigr)\,\vol.
\end{split}
\end{equation*}
By Lemma \ref{estimates} there exists a constant $C>0$ such that
\begin{equation*}
\begin{split}
&|\delbar^*(\delbar a_\nu\land*\varphi)|_g\le C\bigl(|\del\delbar a_\nu|_g|\varphi|_g+|\delbar a_\nu|_g|\nabla\varphi|_g\bigr),\\
&|\delbar^*(\delbar a_\nu\land\varphi)|_g \le C\bigl(|\del\delbar a_\nu|_g|\varphi|_g+|\delbar a_\nu|_g|\nabla\varphi|_g\bigr).
\end{split}
\end{equation*}
Therefore we have
\begin{equation*}
\begin{split}
|I_2(\nu)| & \le C \int_M \bigl(
   |\delbar\delbar^*\varphi|_g(|\del\delbar a_\nu|_g|\varphi|_g+|\delbar a_\nu|_g|\nabla\varphi|_g)+\\
& +|\delbar^*\delbar\varphi|_g(|\del\delbar a_\nu|_g|\varphi|_g+|\delbar a_\nu|_g|\nabla\varphi|_g)+\\
& +|\delbar\varphi|_g|\delbar a_\nu|_g|\varphi|_g+|\del\varphi|_g|\del a_\nu|_g|\varphi|_g\bigr)\,\vol.
\end{split}
\end{equation*}
The estimates on the cut-off functions, i.e., 
$$|\del a_\nu(x)|_g,|\delbar a_\nu(x)|_g,|\del\delbar a_\nu(x)|_g\le 2^{-\nu}C a_\nu(x)^\frac{1}{2},$$
yield
\begin{equation*}
\begin{split}
I_1(\nu)=|I_2(\nu)| & \le 2^{-\nu}C \int_M a_\nu(x)^\frac{1}{2} \bigl(
   |\delbar\delbar^*\varphi|_g(|\varphi|_g+|\nabla\varphi|_g)+\\
& +|\delbar^*\delbar\varphi|_g(|\varphi|_g+|\nabla\varphi|_g)+\\
& +|\delbar\varphi|_g|\varphi|_g+|\del\varphi|_g|\varphi|_g\bigr)\,\vol \\
& \le 2^{-\nu}C \int_M a_\nu(x)^\frac{1}{2}
        \bigl(|\delbar\delbar^*\varphi|_g+|\delbar^*\delbar\varphi|_g+|\delbar\varphi|_g+|\del\varphi|_g\bigr)\cdot \\
& \cdot \bigl(|\varphi|_g+|\nabla\varphi|_g\bigr)\,\vol \\
& \le 2^{-\nu}C \biggl( \int_M a_\nu \bigl(
|\delbar\delbar^*\varphi|_g+|\delbar^*\delbar\varphi|_g+|\delbar\varphi|_g+|\del\varphi|_g\bigr)^2\,\vol\biggr)^{\frac{1}{2}}\cdot\\
& \cdot \biggl( \int_M \bigl(|\varphi|_g+|\nabla\varphi|_g\bigr)^2\,\vol\biggr)^{\frac{1}{2}} \\
& \le 2^{-\nu}C\bigl(I_1(\nu)\bigr)^{\frac{1}{2}}\cdot \bigl(||\varphi||_{L_2}+||\nabla\varphi||_{L_2}\bigr),
\end{split}
\end{equation*}
and consequently
$$
\bigl(I_1(\nu)\bigr)^{\frac{1}{2}}\leq C\,2^{-\nu}.
$$
Thus $\varphi$ is $\del$-closed and $\delbar$-closed. This implies 
$$\tilde{\Delta}_{BC}^g\varphi=\del\delbar\delbar^*\del^*\varphi.$$
Now we substantially reapply the argument as above to this form of the Bott-Chern Laplacian. We have
\begin{equation*}
0 = \la\tilde{\triangle}_{BC}\varphi,a_\nu\varphi\ra = \la\del\delbar\delbar^*\del^*\varphi,a_\nu\varphi\ra = 
\la\delbar^*\del^*\varphi,\delbar^*\del^*(a_\nu\varphi)\ra,
\end{equation*}
and
\begin{equation*}
\begin{split}
\delbar^*\del^*(a_\nu\varphi) & = \delbar^*(-*(\del a_\nu\land*\varphi)+a_\nu\del^*\varphi) \\
  & =(-1)^{p+q-1}*(\delbar\del a_\nu\land*\varphi) -(-1)^{p+q-1}*(\del a_\nu\land\delbar*\varphi)+\\
	& -*(\delbar a_\nu\land*\del^*\varphi)+a_\nu\delbar^*\del^*\varphi,
\end{split}
\end{equation*}
thus
\begin{equation*}
0 = \la\tilde{\Delta}^g_{BC}\varphi,a_\nu\varphi\ra= I'_1(\nu)+I'_2(\nu),
\end{equation*}
where
\begin{equation*}
I'_1(\nu)= \int_M a_\nu|\delbar^*\del^*\varphi|_g^2\,\vol,
\end{equation*}
and
\begin{equation*}
\begin{split}
|I'_2(\nu)| & \le 2^{-\nu}C \int_M a_\nu(x)^\frac{1}{2}
|\delbar^*\del^*\varphi|_g(|\varphi|_g+|\delbar^*\varphi|_g+|\del^*\varphi|_g)\,\vol \\
& \le 2^{-\nu}C \bigl(I'_1(\nu)\bigr)^{\frac{1}{2}}\cdot \bigl(||\varphi||_{L_2}+||\nabla\varphi||_{L_2}\bigr).
\end{split}
\end{equation*}
Thus we also have $\delbar^*\del^*\varphi=0$. This ends the proof.
\end{proof}
\begin{rem}
From Theorem \ref{teo:pezzi} we immediately obtain that 
$$
\varphi\in \mathcal{H}^{p,q}_{BC,2}\quad\hbox{\rm if and only if}\qquad \varphi\in W^{1,2}(M),\qquad \del\varphi=0,\qquad \delbar\varphi=0,
\qquad \delbar^*\del^*\varphi=0,
$$
which ``extends'' the characterization of the space of Bott-Chern harmonic forms on a compact Hermitian manifold to any $d$-bounded Stein manifold.
\end{rem}


As a straightforward consequence of Theorem \ref{teo:pezzi}, we obtain the following

\begin{theorem}
\label{teo:harmonic}
Let $M$ be a d-bounded Stein manifold of complex dimension $n$. Then
\begin{equation*}
\mathcal{H}^{p,q}_{BC,2}\subset\mathcal{H}^{p,q}_{\delbar,2}=\mathcal{H}^{p,q}_{\del,2}=\mathcal{H}^{p,q}_{d,2},
\end{equation*}
where the last three sets are the spaces of $L^2$-harmonic $(p,q)$-forms with respect to $\Delta^g_{\delbar}$, $\Delta^g_{\del}$ and $\Delta^g_{d}$.
\end{theorem}
\begin{proof}
We note that the last three equalities of the thesis hold because $M$ is K\"ahler, in fact $\Delta^g_d=2\Delta^g_{\del}=2\Delta^g_{\delbar}$ by K\"ahler identities. Therefore it is enough to prove $\mathcal{H}^{p,q}_{BC,2}\subset\mathcal{H}^{p,q}_{\delbar,2}$.

Let $\varphi\in\mathcal{H}^{p,q}_{BC,2}$; from (\ref{I1}) in the proof of Theorem 
\ref{teo:pezzi} we have 
$$\delbar\delbar^*\varphi=0,\quad \delbar^*\delbar\varphi=0.
$$ 
Thus $\Delta^g_{\bar\partial}\varphi=0$ and $\varphi\in\mathcal{H}^{p,q}_{\delbar,2}$.
\end{proof}
We are ready to prove the following 
\begin{theorem}\label{GromovBC}
Let $M$ be a $d$-bounded Stein manifold of complex dimension $n$. Then 
$$\mathcal{H}^{p,q}_{BC,2}=\{0\},\quad \hbox{for}\,\, p+q\neq n.
$$
\end{theorem}
\begin{proof} By Theorem \ref{teo:harmonic}, it is enough to prove that 
$\mathcal{H}^{p,q}_{d,2}=\{0\}$.
This last fact is a consequence of \cite[Thm.1.2.B.]{G}. For the sake of completeness we remind the argument by Gromov. 

Let us consider the Lefschetz operator 
$$
L:\Omega^{(p,q)}(M)\lr\Omega^{(p+1,q+1)}(M)
$$ defined by 
$$L\varphi=\omega\land\varphi
$$
for every $\varphi\in\Omega^{(p,q)}(M)$. By \cite[rem.3.2.7.iii)]{H}, the map
\begin{equation*}
L^{n-p-q}:\mathcal{H}^{p,q}_{d}\lr\mathcal{H}^{n-q,n-p}_{d}
\end{equation*}
is an isomorphism for $p+q\le n$, where $\mathcal{H}^{p,q}_{d}$ denotes the space of 
$\Delta^g_{d}$-harmonic $(p,q)$-forms. Now set $k=n-p-q$ and consider the form $L^k\varphi$, where $\varphi\in\Omega^{(p,q)}(M)\cap L^2(M)$ 
is a $d$-closed form. Since $\omega=d\eta$, if $k>0$, then we get
\begin{equation*}
L^k\varphi=\omega^k\land\varphi=(d\eta)^k\land\varphi=d(\eta\land(d\eta)^{k-1}\land\varphi).
\end{equation*}
Furthermore, 
\begin{itemize}
\item $\eta\land(d\eta)^{k-1}$ is bounded, since $\eta$ is bounded and $|d\eta|_g$ is constant;
\item $\eta\land(d\eta)^{k-1}\land\varphi\in L^2(M)$, since $\varphi\in L^2(M)$.
\end{itemize}

Moreover, if $\varphi\in\Omega^{(p,q)}(M)\cap L^2(M)$ is a $\Delta^g_d$-harmonic form, 
then $L^k\varphi$ is also $\Delta^g_d$-harmonic. Thus, in view of the $L^2$ 
Hodge decomposition theorem (see \cite[Chap.VIII, Thm.3.2]{DE}), we obtain that $L^k\varphi=0$.

Now let $\varphi\in\mathcal{H}^{p,q}_{d,2}$. If $p+q<n$, then $k>0$ and $L^k\varphi=0$; therefore $\varphi=0$ since $L^k$ is injective.
Conversely, if $p+q>n$, then $*\varphi\in\mathcal{H}^{n-p,n-q}_{d,2}$ and $(n-p)+(n-q)<n$; by the previous argument $*\varphi=0$ 
and consequently $\varphi=0$. 

Summing up, we showed that $\mathcal{H}^{p,q}_{d,2}=\{0\}$ for $p+q\ne n$. This ends the proof.
\end{proof}
\begin{rem}
The Lefschetz argument, in the proof of Theorem \ref{GromovBC}, uses only the K\"ahler and the $d$-bounded assumptions.
\end{rem}
\begin{rem} By  similar computations, Theorem \ref{GromovBC} can be stated and proved also for the Aeppli Laplacian. 
Indeed, it is sufficient to repeat the proof of Theorem \ref{teo:pezzi} 
with $\tilde\Delta_{A}^g=\Delta^g_{\bar\partial}\Delta^g_{\bar\partial}+\del\del^*+\delbar\delbar^*$.
\end{rem}

\end{document}